\documentclass[11pt,leqno]{amsart}
\usepackage{amsmath,amsthm, amsfonts,amssymb}
\usepackage{enumerate}
\usepackage{hyperref}
\usepackage[T1]{fontenc}
\usepackage{cite}
\usepackage{color}
\usepackage{soul}
\newtheorem{thm}{Theorem}[section]
\newtheorem{lem}[thm]{Lemma}

\newtheorem{cor}[thm]{Corollary}

\theoremstyle{definition}

\theoremstyle{remark}
\newtheorem*{remark}{Remark}

\newcommand{\la}{\lambda}

\newcommand{\bH}{\mathbb{H}}
\newcommand{\ve}{\varepsilon}

\DeclareMathOperator{\supp}{supp}

\title[Strong continuity on Hardy spaces]{Strong continuity on Hardy spaces}
\author{Jacek Dziuba\'nski and B\l a\.{z}ej Wr\'obel}
\address{
Instytut Matematyczny, Uniwersytet Wroc\l awski, pl. Grunwaldzki 2/4, 50-384 Wroc\l aw, Poland}
\email{jdziuban@math.uni.wroc.pl}
\address{Dipartimento di Matematica e Applicazioni, Universit\`{a} di Milano-Bicocca,
via R. Cozzi 53 I-20125, Milano, Italy,
\newline \&
Instytut Matematyczny, Uniwersytet Wroc\l awski, pl. Grunwaldzki 2/4, 50-384 Wroc\l aw, Poland}
\email{blazej.wrobel@math.uni.wroc.pl}
 
\begin{document}

 \begin{abstract}
We prove the strong continuity of spectral multiplier operators associated with dilations of certain functions  on the general Hardy space $H^1_L$ introduced by Hofmann, Lu, Mitrea, Mitrea, Yan. Our results include the heat and Poisson semigroups as well as the group of imaginary powers.
\end{abstract}

  \maketitle
 \numberwithin{equation}{section}

\section{Introduction}

In the theory of semigroups of linear operators on Banach spaces the crucial assumption is that of strong continuity. One often encounters a situation where the semigroup $T_t=e^{-tL}$ is initially defined on $L^2(\Omega)$ and $L$ is a non-negative self-adjoint operator. In this case the spectral theorem immediately gives the strong $L^2(\Omega)$ continuity $\lim_{t\to 0^+}\|T_t f-f\|_{L^2(\Omega)}=0,$ for $f\in L^2(\Omega)$. Assume additionally that $\{T_t\}_{t>0}$ extends to a locally bounded semigroup on $L^p.$ More precisely, we impose that for each $1\leq p<\infty$ there exists $t_p>0$ such that $\|T_t \|_{L^p(\Omega)\to L^p(\Omega)}\leq C_p,$ $t\in[0,t_p].$ Since weak and strong convergence coincide for semigroups of operators (see e.g.\ \cite[Theorem 5.8]{En_Nag_1}), it is  straightforward to see that $T_t$ is strongly continuous on all $L^p(\Omega),$ $1 <p<\infty.$ Moreover, if we assume that $\{T_t\}_{t>0}$ is contractive on $L^1(\Omega),$ then it is also strongly continuous on $L^1(\Omega).$ Quite often the semigroup $\{T_t\}_{t>0}$ may be also defined on function spaces other than $L^p.$ For instance, if $T_t=e^{t\Delta}$ is the classical heat semigroup on $\mathbb{R}^d,$ then it also acts on the atomic Hardy spaces $H^1_{at}.$ However, even in this case it is not obvious that the semigroup is strongly continuous on $H^1_{at}.$

In this paper we impose that $\{T_t\}_{t>0}$ satisfies the so-called Davies-Gaffney estimates (see \eqref{eq:D-G}), and that the underlying space $\Omega$ is a space of homogeneous type in the sense of Coifman-Weiss
\cite{CW}. Under these assumptions,  as a corollary of our main result, we prove that  $e^{-tL}$ and $e^{-t\sqrt{L}}$  are  strongly continuous on the Hardy space $H_L^1.$ This Hardy space was introduced by Hofmann, Lu, Mitrea, Mitrea, Yan  in \cite{HofLMiMiYa1}. Our results are quite general, as there are many operators $L$  satisfying \eqref{eq:D-G}, e.g. Laplace-Beltrami operators on complete Riemannian manifolds (see e.g. \cite[Corollary 12.4]{Gri1}) or  Schr\"odinger operators with non-negative potentials.

The literature on $L^p$ spectral multipliers for operators satisfying Davies-Gaffney estimates is vast. However, as the $L^p$ theory is not discussed in our paper, we do not provide detailed references on this subject. Instead we kindly refer the interested reader to consult e.g. \cite{SiYanYao1} and references therein. There are also results for spectral multipliers on the Hardy space $H_L^1$ (or more generally $H_L^p$),  see e.g.\ \cite{DuLi1}, \cite{DuYan1}, \cite{DzPr_Arg}, and \cite{KunUhl}.

The methods we use are based on \cite{DzPr_Arg}, in which the authors proved a H\"ormander-type multiplier theorem on $H_L^1$. The result for semigroups (Corollary \ref{cor:StrHeat}) is a consequence of Theorem \ref{thm:dilgen}, which treats dilations of more general multipliers than $e^{-\la}.$ Finally, using Theorem \ref{thm:dilgen} we also prove the strong $H^1_L$ continuity of the group of imaginary powers
$\{L^{iu}\}_{u\in \mathbb{R}},$ see Corollary \ref{cor:StrIma}.

\section{Preliminaries}
Let $(\Omega, \, d(x,y))$ be a metric space equipped with a
positive measure $\mu$. We assume that $(\Omega, \, d, \, \mu)$ is
a space of homogeneous type in the sense of Coifman-Weiss
\cite{CW}, that is, there exists a constant $C>0$ such that
 \begin{equation}\label{eq:doubling}
 \mu(B_d(x,2t))\leq C\mu (B_d(x,t)) \ \ \
 \text{for every} \ x\in\Omega, \ t>0,
 \end{equation}
 where $B_d(x, t)=\{ y\in\Omega: \ d(x,y)<t\}$.
  The condition (\ref{eq:doubling}) implies that there exist  constants
   $C_0>0$ and $q>0$ such that
  \begin{equation}\label{eq:growth}
 \mu(B_d(x,st))\leq C_0 s^q \mu(B_d(x,t)) \ \ \ \text{for every }
 x\in \Omega, \ t>0, \ s>1.
 \end{equation}
In what follows we set $n_0$ to be the infimum over $q$ in \eqref{eq:growth}.

   Let   $\{e^{-tL}\}_{t>0}$ be a semigroup of linear operators on $L^2(\Omega, \, d\mu)$
   generated by  $-L$, where $L$ is a non-negative, self-adjoint
    operator. We assume additionally that $L$ is injective on its domain.  Throughout the paper we impose that $T_t:=e^{-tL}$
   satisfies
 Davies-Gaffney estimates, that is,
 \begin{equation}\label{eq:D-G} |\langle T_t f_1,f_2\rangle |\leq C
 \exp\left(-\frac{\text{dist}(U_1,U_2)^2}{ct}\right)\|
 f_1\|_{L^2(\Omega)}\| f_2\|_{L^2(\Omega)}
 \end{equation}
 for every $f_i\in L^2(\Omega)$, $\text{supp}\, f_i\subset U_i$,
 $i=1,2$, $U_i$ are open subsets of $\Omega.$

  Davies-Gaffney estimates are equivalent to the finite speed propagation of the wave equation; the reader interested in this topic is kindly referred to \cite{CouSi1}. The finite speed propagation of the wave equation is used in the proof of \cite[Lemma 4.8]{DzPr_Arg} (our Lemma \ref{lem:Lem48}), which is an important ingredient in the proof of our main Theorem \ref{thm:dilgen}.

%Sometimes we need to be more restrictive. We assume that the operators $T_t:=e^{-tL}$ have
%     the form
%   \begin{equation}\label{eq:sem}
% T_tf(x)=\int_\Omega T_t(x,y)f(y)d\mu (y),
%   \end{equation}
%   where  the kernels $T_t(x,y)$ satisfy  Gaussian bounds,
%   that is, there exist constants $C_0$, $c_0>0$ such that for
%   every $x,y\in\Omega$, $t>0$, we have
%   \begin{equation}\label{eq:G1}
%   |T_t(x,y)|\leq \frac{C_0}{V(x,\sqrt{t})}
%   \exp\left(-\frac{d(x,y)^2}{c_0t}\right),
%   \end{equation}
% where  here and subsequently $V(x,t)=\mu (B_d(x,t))$.
% The estimate (\ref{eq:G1}) implies that for every $k\in\mathbb N$
% there exist  constants  $C_k,c_k>0$ such that
% \begin{equation}\label{eq:G2}
%   \left|\frac{\partial^k}{\partial t^k}T_t(x,y)\right|
%   \leq \frac{C_k}{t^k V(x,\sqrt{t})}
%   \exp\left(-\frac{d(x,y)^2}{c_k t}\right) \ \ \text{  for } \
% x,y\in\Omega, \ t>0.
%   \end{equation}
%   The constants $C_k, \ c_k$  in (\ref{eq:G2}) depend only on $k$ and the constants
%   $C, C_0, q,c$ in (\ref{eq:growth}) and (\ref{eq:G1}).

For $f\in L^2(\Omega)$ we consider the
 square  function $S_hf$ associated with $L$   defined by
\begin{equation*}
%\label{eq:square}
 S_hf(x)=\left(\iint_{\Gamma (x)}|t^2L T_{t^2}
 f(y)|^2\frac{d\mu(y)}{V(x,t)}\frac{dt}{t}\right)^{1\slash 2},
 \end{equation*}
where $\Gamma (x)=\{(y,t)\in\Omega \times (0,\infty):\ d(x,y)\leq
t\}$.

We define the Hardy space $H^1_L=H^1_{L,
S_h}(\Omega)$ as the (abstract) completion of
$$\{f\in L^2(\Omega):\ \|
S_hf\|_{L^1(\Omega)}<\infty\}$$
 in the norm $\| f\|_{H^1_L}=\|S_h
f\|_{ L^1(\Omega)}$.

 It was proved in Hofmann, Lu, Mitrea, Mitrea, Yan \cite{HofLMiMiYa1} that under our assumption \eqref{eq:D-G}
   the space $H^1_L$
    admits the following atomic
 decomposition.

  Let $M\geq 1$, $M\in \mathbb N$. A function $a$ is a
  $(1,2,M)$-atom for
  $H^1_L$  if there exist a ball
  $B=B_d(y_0,r)=\{ y\in \Omega: \, d(y,y_0)<r\}$ and a function
  $b\in \mathcal D(L^M)$  such that
  \begin{equation*}
%\label{A1}
 a=L^Mb;\end{equation*}
  \begin{equation*}
%\label{A2}
\text{supp}\, L^k b \subset B,  \ \
  k=0,1,...,M ; \end{equation*}
  \begin{equation*}
%\label{A3}
\| (r^2L)^kb\|_{L^2(\Omega)}
  \leq r^{2M}\mu(B)^{-1\slash 2}, \ \ k=0,1,...,M.
  \end{equation*}
We say that $f=\sum_j\lambda_j a_j$ is a $(1,2,M)$ atomic representation (of $f$) if $\{\la_j\}_{j=0}^{\infty}\in l^1,$ each $a_j$ is a $(1,2,M)$ atom, and the sum converges in $L^2.$ Then we set
\begin{align*}
\bH^1_{L,at,M}=\bigg\{f\colon \textrm{$f$ has an atomic $(1,2,M)$-representation}\bigg\},
\end{align*}
with the norm given by
\begin{align*}
\|f\|_{\bH^1_{L,at,M}}=\inf\bigg\{\sum_{j=0}^{\infty} |\la_j|\colon f=\sum_{j=0}^{\infty}\lambda_j a_j\textrm{ is an atomic $(1,2,M)$ representation}\bigg\}.
\end{align*}
The space $H^1_{L,at,M}$ is defined as the (abstract) completion of $\bH^1_{L,at,M}.$

Theorem 4.14 of \cite{HofLMiMiYa1}
asserts that for each $M> n_0/4$ there exists a constant $C>0$ such that
\begin{equation*}
%\label{dz}
 C^{-1} \| f\|_{H^1_L}\leq  \| f\|_{H^1_{L,at,M}}\leq C\|
 f\|_{H^1_L}.
\end{equation*}

In \cite{HofLMiMiYa1} the authors gave also a molecular description of $H^1_L.$ Fix $\varepsilon >0$ and $M>n_0\slash 4$, $M\in\mathbb N$.
  We say that a function $\tilde a$ is a
 $(1,2,M,\varepsilon )$-molecule associated to $L$ if
 there exist a function $\tilde b\in \mathcal D(L^M)$ and
 a ball $B=B_d(y_0,r)$ such that
 \begin{equation*}
%\label{molec1}
  \tilde a=L^M \tilde b;
  \end{equation*}
 \begin{equation*}
%\label{molec2}
\|(r^2L)^k\tilde b\|_{L^2(U_jB))}
 \leq r^{2M }2^{-j\varepsilon } \mu(B(y_0, 2^jr))^{-1\slash  2}
 \end{equation*}
for  $ k=0,1,...,M$,  $j=0,1,2,...$, where
$U_0=B$, $U_j(B)=B_d(y_0,2^jr)\setminus B_d(y_0, 2^{j-1}r)$ for $j\geq 1$. The decomposition $f=\sum_j\lambda_j \tilde a_j$ is a $(1,2,M,\ve)$ molecular representation (of $f$) if $\{\la_j\}_{j=0}^{\infty}\in l^1,$ each $\tilde a_j$ is a $(1,2,M,\ve)$ molecule, and the sum converges in $L^2.$ Then we define
\begin{align*}
\bH^1_{L,mol,M,\ve}=\bigg\{f\in L^2(\Omega)\colon \textrm{$f$ has a molecular $(1,2,M,\ve)$-representation}\bigg\},
\end{align*}
with the norm given by
\begin{align*}
\|f\|_{\bH^1_{L,mol,M,\ve}}=\inf\bigg\{\sum_{j=0}^{\infty} |\la_j|\colon f=\sum_{j=0}^{\infty}\lambda_j \tilde a_j\textrm{ is a molecular $(1,2,M,\ve)$ representation}\bigg\}.
\end{align*}
The space $H^1_{L,mol,M,\ve}$ is defined as the (abstract) completion of $\bH^1_{L,mol,M}.$

 It was proved in \cite[Corollary 5.3] {HofLMiMiYa1} that for each $M>n_0/4$ and $\ve>0$ it holds $\bH^1_{L,at,M}=\bH^1_{L,mol,M,\ve},$
with the equivalence of the norms. Moreover, we have $H_L^1=H^1_{L,at,M}$ and, consequently, $H_L^1=H^1_{L,at,M}=H^1_{L,mol,N,\ve},$ for $N,M > n_0/4.$

The following lemma is a slight extension of the observation following the proof of \cite[Corollary 5.3] {HofLMiMiYa1}.
\begin{lem}
\label{lem:EnoughOnAtoms}
Let $T$ be an operator which is bounded on $L^2.$ Assume that there are  $\ve>0$ and positive integers $M,N >n_0/4$ such  that $T$ maps $(1,2,M)$ atoms uniformly to $(1,2,N,\ve)$ molecules. More precisely, we impose that there is an $A>0$ such that  $\|T(a)\|_{\bH^1_{L,mol,N,\ve}}\leq A \|a\|_{\bH^1_{L,at,M}}$ for all $(1,2,M)$ atoms $a$. Then $T$ has the unique bounded extension $T^{ext}$ to $H_{L}^1$ which satisfies
\begin{equation*}
%\label{eq:extbound}
\|T^{ext}f\|_{H_L^1}\leq C\, A \|f\|_{H_L^1}.
\end{equation*}
\end{lem}
\begin{proof}
By density of $\bH^1_{L,at,M}$ in $H^1_{L,at,M}=H_L^1$ it is enough to prove that $T$ is bounded from $\bH^1_{L,at,M}$ to $\bH^1_{L,mol,N,\ve}.$

Take $f\in \bH^1_{L,at,M},$ so that  $f=\sum_j \la_j a_j,$ where $a_j$ are $(1,2,M)$ atoms, $\{\la_j\}\in l^1,$ and the sum converges in $L^2.$ We chose $\la_j$ and $a_j$ in a way that $\sum_{j}|\la_j|\leq 2\|f\|_{\bH^1_{L,at,M}}.$ The $L^2$ boundedness of $T$ implies that $Tf=\sum_j \la_j T(a_j)$ is a $(1,2,N,\ve)$ molecular representation of $Tf.$ Therefore,
$$\|Tf\|_{\bH^1_{L,mol,N,\ve}}\leq A\,\sum_{j}|\la_j|\leq 2A\,\|f\|_{\bH^1_{L,at,M}},$$
and the proof is completed.
\end{proof}

  Let $E:=E_{\sqrt{L}}$ be the spectral measure of $\sqrt{L}$ so that
 \begin{equation*}
%\label{eq:res}
 Lf=\int_0^\infty \lambda^2 \, dE(\lambda)f.
 \end{equation*}
Then, for a bounded Borel-measurable function $m\colon [0,\infty)\to \mathbb{C}$ the spectral multiplier operator $m(\sqrt{L})$ is given on $L^2(\Omega)$ by
\begin{equation*}
\label{eq:spectmult}
 m(\sqrt{L})f=\int_0^\infty m(\lambda) \, dE(\lambda)f.
 \end{equation*}

Using Lemma \ref{lem:EnoughOnAtoms} with $2M$ in place of $M$ and $N=M>n_0/4$ we deduce the following enhancement of
\cite[Theorem 4.2]{DzPr_Arg}.
 \begin{thm}\label{thm:DzPr}
Assume  that $m$ is  a bounded function defined on $[0,\infty)$ and such
that for  some real number $\alpha > (n_0+1)\slash 2$ and any
nonzero function $\eta\in C_c^\infty (2^{-1}, 2)$ we have
\begin{equation}\label{condition}
 \|m\|_{\eta,\alpha}:=\sup_{t>0} \| \eta (\, \cdot \,) m(t\,\cdot \,)\|_{W^{2
,\alpha}(\mathbb R)}<\infty,
\end{equation}
where $\| F \|_{W^{p,\alpha}(\mathbb R)}=\| (I-d^2\slash
dx^2)^{\alpha \slash 2}F\|_{L^p(\mathbb R)}.$
  Then the operator $m(\sqrt{L})$ extends uniquely to a bounded operator on $H_L^1.$ Moreover, there exists a constant $C>0$ such that
  \begin{equation*}
%\label{eqqq2}
 \| m(\sqrt{L})f\|_{H^1_{L}}\leq C \|m\|_{\eta,\alpha} \|f\|_{H^1_L}, \qquad f\in H^1_{L}.
  \end{equation*}
\end{thm}

For the convenience of the reader we also restate Lemma 4.8 of \cite{DzPr_Arg}.
 \begin{lem}\label{lem:Lem48}
 Let $\gamma >1\slash 2$, $\beta >0$. Then there exists a constant $C>0$ such that
 for every even function  $F\in W^{2,\gamma
 +\beta\slash 2}(\mathbb R)$ and every  $g\in
 L^2(\Omega)$, $\text{\rm supp}\, g\subset B_d(y_0, r)$,
 we have
  $$ \int_{d(x,y_0)>2r} |F(2^{-j}\sqrt{
  L})g(x)|^2\left(\frac{d(x,y_0)}{r}\right)^\beta d\mu(x)\leq C
  (r2^j)^{-\beta}
  \| F\|_{W^{2,\gamma +\beta\slash 2}}^2\| g\|_{L^2(\Omega)}^2$$
  for $j\in\mathbb Z$.
 \end{lem}

Summarizing this section, we may use whichever of the spaces $H^1_{L,at,M}$ or $H^1_{L,mol,M},$  $M> n_0/4,$ that is convenient.

\section{The results}
We are going to study strong $H^1_L$ convergence of operators of the form $m(tL)$ as $t\to 0$. Observe that for the strong $L^2$ convergence it is enough to assume that $m$ is bounded and continuous at $0.$ Our first main result is the following theorem.
 \begin{thm}
\label{thm:dilgen}
 Take $\kappa$ an integer larger than $(n_0+1)\slash 2.$ Let $m\colon [0,\infty)\to \mathbb{C}$ be a continuous function which is $C^{\kappa}$ on $(0,\infty).$ Assume that $m$ satisfies the Mikhlin condition of order $\kappa,$ i.e.
 \begin{equation}
 \label{eq:Mikh_cond}
 \sup_{0\leq j\leq \kappa}\sup_{\la>0}|\la^j m^{(j)}(\la)|<\infty,
 \end{equation}
 and, additionally
  \begin{equation}
 \label{eq:Mikh_lim0}
 \lim_{\la\to 0^+}\la^j m^{(j)}(\la)=0,\qquad j=1,\ldots,\kappa.
 \end{equation}
Then, we have the following strong $H_L^1$ convergence,
\begin{equation}\label{stronconv}
 \lim_{t\to 0^+}m(t\sqrt{L}) f= m(0)f,\qquad \textrm{for every }f\in H_L^1.
 \end{equation}

\end{thm}

\begin{remark}
Straightforward modifications in the proof we present below give a slightly stronger version of the theorem, with the assumption \eqref{eq:Mikh_cond} replaced by \eqref{condition} for some real number $\alpha$ larger than $(n_0+1)/2.$
\end{remark}

Before proceeding to the proof let us note the following important corollary.
\begin{cor}
\label{cor:StrHeat}
Both the heat semigroup $e^{-tL}$ and the Poisson semigroup $e^{-t\sqrt{L}}$ are strongly continuous on $H^1_L.$
\end{cor}

\begin{proof}[Proof of Theorem \ref{thm:dilgen}]
 Let $M$ be an integer such that $2M\geq \kappa$. Then $M>n_0\slash 4$. From Theorem \ref{thm:DzPr} and the dilation invariance of \eqref{condition} it follows that $m(tL)$ is well-defined and bounded on $H_L^1,$ uniformly in $t>0.$ Therefore it is enough to prove \eqref{stronconv} for $f\in \bH^1_{L,at,2M}.$

We claim that we can further reduce the proof to demonstrating that
\begin{equation}
\label{stronconvat}
\lim_{t\to 0^+}\|m(t\sqrt{L})a-m(0)a\|_{H_L^1}=0,\qquad \textrm{for $a$ being a $(1,2,2M)$-atom}.
\end{equation}
Indeed, if \eqref{stronconvat} is true, and $f=\sum_{j} \la_j a_j$ (where $\{\la_j\}\in l^1$ and the sum defining $f$ converges also in $L^2$) then we obtain
\begin{align*}
\|[m(t\sqrt{L})-m(0)](f)\|_{H_L^1}= \big\|\sum_{j=0}^{\infty}\la_j[m(t\sqrt{L})-m(0)](a_j)\big\|_{H_L^1}\leq \sum_{j=0}^{\infty}|\la_j|\|[m(t\sqrt{L})-m(0)](a_j)\|_{H_L^1}.
\end{align*}
Now, from Theorem \ref{thm:DzPr} it follows that $\|[m(t\sqrt{L})-m(0)](a_j)\|_{H_L^1}$ is uniformly bounded in $t.$ Therefore, thanks to  \eqref{stronconvat} we obtain $\lim_{t\to 0^+} \|[m(t\sqrt{L})-m(0)](f)\|_{H_L^1}=0,$ as desired.

To prove \eqref{stronconvat} we will show that there is an $\varepsilon >0$ such that for every $a$ being a $(1,2,2M)$-atom the function $(m(t\sqrt{L})-m(0)) a$ is a multiple of a $(1,2,M,\varepsilon)$ molecule and the multiple constant tends to $0$ as $t\to0$.  Note that the rate of convergence may well depend on $a$ for our purposes. There is no loss of generality if we assume that the associated ball $B$ has radius 1, that is $B=B(y_0,1)$ for certain $y_0\in\Omega$. This means that $a=L^{2M}b$ where $b\in \mathcal D(L^{2M})$, $\supp L^k b\subset B,$ and $\| L^kb\|_{L^2}\leq 1$ for $k=0,1,...,2M$. Then, denoting $\tilde b=[m(t\sqrt{L})-m(0)]L^M b,$ we have $[m(t\sqrt{L})-m(0)] a=L^M \tilde b$. Our task is to study the behavior of $L^2$-norms of $L^k \tilde b=[m(t\sqrt{L})-m(0)]L^{k+M}b$, $k=0,1,...,M$,  on the sets $U_j(B)$.

Let $\psi\in C_c^\infty (\frac{1}{2},2)$ be such that $\sum_{\ell \in\mathbb Z}\psi (2^{-\ell}\lambda)=1$ for $\lambda >0$. For $\ell_0\in \mathbb Z$ and $\lambda\in \mathbb R$ set $\Psi_{\ell_0}(\lambda)=1-\sum_{\ell=\ell_0}^\infty \psi (2^{-\ell}|\lambda|)$. We split $$m(t|\la|)-m(0)=\Psi_{\ell_0} (\lambda) (m(t|\la|)-m(0))+\sum_{\ell=\ell_0}^\infty \psi(2^{-\ell}|\lambda|)(m(t|\la|)-m(0))$$
and for $\lambda\in\mathbb R$  put
$$m_{\ell,t} (\lambda)=\psi(2^{-\ell}|\lambda|)(m(t|\la|)-m(0)), \ \ \tilde m_{\ell,t} (\lambda)=m_{\ell,t} (2^{\ell} \lambda)=\psi(|\lambda|)(m(t2^{\ell}|\lambda|)-m(0)).$$

Fix $\ve>0$ and $\gamma>1/2$ such that $\gamma+\ve+n_0/2=\alpha.$  Set $\beta=n_0+2\ve,$ so that $\gamma+\beta/2=\alpha.$
Recall that $\supp L^{k+M} b\subset B$ and $m_{\ell,t} (\lambda)=m_{\ell,t} (-\lambda)$. Applying Lemma \ref{lem:Lem48} we have
$$ \int_{d(x,y_0)>2}| m_{\ell, t}(\sqrt{L})L^{k+M}b(x)|^2d(x,y_0)^\beta \, d\mu (x)\leq C 2^{-\ell \beta}\| \tilde m_{\ell, t}\|_{W^{2,\alpha}}\| L^{k+M}b\|_{L^2}^2,$$
 hence, using \eqref{eq:Mikh_cond} we arrive at
$$  \int_{U_j(B)}| m_{\ell, t}(\sqrt{L})L^{k+M}b(x)|^2 \, d\mu (x)\leq C_{\alpha} 2^{-\ell \beta}2^{-j\beta}\| L^{k+M}b\|_{L^2}^2.$$
Therefore
\begin{equation}
\label{eq:nearinf}
\begin{split}
\Big(\int_{U_j(B)} \Big|\sum_{\ell >\ell_0} m_{\ell, t}(\sqrt{L}) L^{k+M}b|^2\, d\mu \Big)^{1\slash 2}\leq  C_{\alpha}^{1\slash 2} 2^{-j\beta\slash 2} 2^{-\ell_0 \beta\slash 2} \| L^{k+M}b\|_{L^2}.
\end{split}\end{equation}
Note that the estimate above does not depend on $t>0$. For the rest of the proof we fix $\ell_0$ large enough.

Denote $n_{\ell_0, t}(\lambda)=\Psi_{\ell_0} (\lambda)  (m(t|\la|)-m(0))\la^{2M}$,  $\lambda\in\mathbb R$. Clearly, $n_{\ell_0, t}(\lambda)=n_{\ell_0, t}(-\lambda)$. Using Lemma \ref{lem:Lem48} we get
$$ \int_{d(x,y_0)>2} |L^{k}n_{\ell_0,t}(\sqrt{L}) b(x)|^2 d(x,y_0)^\beta d\mu  (x)\leq C\| n_{\ell_0 ,t}\|_{W^{2,\gamma+\beta\slash 2}}^2 \| L^{k}b\|_{L^2}^2 = C\| n_{\ell_0 ,t}\|_{W^{2,\alpha}}^2 \| L^{k}b\|_{L^2}^2,$$
and, consequently,
\begin{equation}
\label{eq:nearzer}
\int_{U_j(B)} |L^{k}n_{\ell_0,t}(\sqrt{L}) b(x)|^2\, d\mu (x)\leq C 2^{-\beta j} \| n_{\ell_0 ,t}\|^2_{W^{2,\alpha}} \| L^{k}b\|_{L^2}^2.\end{equation}
 We claim that $n_{\ell_0, t}(\lambda)=\Psi_{\ell_0} (\lambda)  (m(t|\la|)-m(0))\la^{2M}$ satisfies $\lim_{t\to 0^+}\| n_{\ell_0 ,t}\|_{W^{2,\alpha}}=0$. Indeed
\begin{align*}
\| n_{\ell_0 ,t}\|_{W^{2,\alpha}}&\leq \| n_{\ell_0 ,t}\|_{W^{2,{\kappa}}}\approx \|n_{\ell_0 ,t}\|_{L^2}+\| (n_{\ell_0 ,t})^{(\kappa)}\|_{L^2}\lesssim C_{l_0}\|(m(t|\la|)-m(0))\la^{2M}\|_{C^{\kappa}[0,2^{l_0+1}]},
\end{align*}
and, because of \eqref{eq:Mikh_lim0}, the quantity on the right hand side of the above inequality approaches $0$ as $t\to 0^+.$
Summarizing \eqref{eq:nearinf} and \eqref{eq:nearzer} we have proved that, for $k=0,\ldots,M,$ it holds
 \begin{equation*}
 \begin{split}
 \int_{U_j(B)} |L^{k+M} [m(t\sqrt{L})-m(0)]b(x)|^2 \, d\mu (x)
 &\leq C 2^{-j\beta} (2^{-\ell_0\beta}\| L^{k+M}b\|_{L^2}^2+ \| n_{\ell_0 ,t}\|_{W^{2,\alpha}}^2 \| L^kb\|_{L^2}^2)\\
 & \leq C 2^{-j\beta} (2^{-\ell_0\beta}+ \| n_{\ell_0 ,t}\|_{W^{2,\alpha}}^2 )\mu(B(y_0,1))^{-1}\\
 & \leq C 2^{-j\beta} (2^{-\ell_0\beta}+ \| n_{\ell_0 ,t}\|_{W^{2,\alpha}}^2 )\frac{\mu(B(y_0,2^j))}{\mu(B(y_0,1))}\mu(B(y_0,2^j))^{-1}.\\
 \end{split}\end{equation*}
Using \eqref{eq:growth} with $q=n_0+\varepsilon$ we obtain
 \begin{equation*}
 \begin{split}
 \int_{U_j(B)} |L^{k+M} [m(t\sqrt{L})-m(0)]b(x)|^2 \, d\mu (x)
  & \leq C 2^{-j\beta} (2^{-\ell_0\beta}+ \| n_{\ell_0 ,t}\|_{W^{2,\alpha}}^2 )2^{jq} \mu(B(y_0,2^j))^{-1},
  \end{split}\end{equation*}
which is enough for our purpose, since  $\beta- q=\ve,$ $\gamma+\beta\slash 2=\alpha,$ and $\lim_{t\to 0^+}\| n_{\ell_0 ,t}\|_{W^{2,\alpha}}^2 =0$.

To estimate $L^{k+M}(m(t\sqrt{L})-m(0))b$ on $2B$, we note that by the spectral theorem,
\begin{align*}
\| [m(t\sqrt{L})-m(0)] L^{k+M}b\|_{L^2(2B)}^2 &\leq \| [m(t\sqrt{L})-m(0)] L^{k+M}b\|_{L^2(\Omega)}^2\\
&=\int_0^\infty |m(t\la)-m(0)|^2 dE_{L^{k+M}b,L^{k+M}b}(\lambda)\to 0 \ \ \text{as } t\to 0
\end{align*}
thanks to  the Lebesgue dominated convergence theorem and the continuity of $m$ at $0$.
\end{proof}

We finish the paper with showing the strong convergence of the group of imaginary powers. This is achieved by using Theorems \ref{thm:DzPr} and \ref{thm:dilgen}.
\begin{cor}
 \label{cor:StrIma}
 Let $f\in H^1_L.$ Then $\lim_{u\to 0}L^{iu}f=f,$ the limit being in $H^1_L.$
\end{cor}
\begin{proof}
Let $\phi(\la)$ be a smooth function on $[0,\infty)$ which is equal to $1$ on $[0,2]$ and vanishes for $\la>4.$

Theorem \ref{thm:DzPr} implies
\begin{equation}
\label{eq:UniBound}
 \sup_{|u|\leq 1}\|L^{iu}\|_{H^1_L\to H^1_L}<\infty.
\end{equation}
Moreover, from Theorem \ref{thm:dilgen} it follows that $\lim_{s\to 0^+} \phi(sL)f=f,$ for $f\in H_L^1$ (the limit being in $H_L^1$).
Hence, a density argument together with \eqref{eq:UniBound} show that it is enough to justify that for each fixed $s>0,$ we have
\begin{equation}
\label{eq:ImaConvExp}
 \lim_{u\to 0}(L^{iu}-1)\phi(sL)f=0,\qquad f\in H^1_L,
\end{equation}
the limit being understood in $H^1_L.$  Let $M$ be an integer larger than $(n_0+3)/2$. As the linear span of atoms is dense in $H^1_L,$ in view of \eqref{eq:UniBound} it suffices to verify \eqref{eq:ImaConvExp} for $f$ being a fixed $(1,2,2M)$ atom.
Then $f=L^{2M}b.$ Moreover, $a=L^{M}b$ is a multiple of a $(1,2,M)$ atom, with a multiple constant that depends on $f.$ Let  $m_u(\la)=\la^{2M}(\la^{2iu}-1)\phi(s\la^2)$ and let $\eta$ be a non-zero smooth function supported in $[1/2,2].$
A short computation shows that $$\lim_{u\to 0}\sup_{t>0}\|\eta(\cdot) m_u(t\cdot )\|_{W^{2,M-1}}=0.$$ We also have  $(L^{iu}-1)\phi(sL)f=m_u(\sqrt{L})(a)$ with $a$ being a $(1,2,M)$-atom. Since $M-1>(n_0+1)\slash 2,$  using Theorem \ref{thm:DzPr} we finish the proof of Corollary \ref{cor:StrIma}.
\end{proof}
\begin{remark}
Corollary \ref{cor:StrIma} seems crucial in extending various results in harmonic analysis based on the group of imaginary powers from the $L^p$ to the $H_L^1$ setting. For potential applications see e.g. \cite{Meda_gmt} or \cite[Remark 3]{Wrobel_cMH}.
\end{remark}

\subsection*{Acknowledgments}
The research of the first named author was supported by Polish funds for sciences, National Science Centre (NCN), Poland, Research Project
DEC-2012\slash 05\slash B\slash ST1\slash 00672. The research of the second named author was supported by Polish funds for sciences, National Science Centre (NCN), Poland, Research Project DEC-2014\slash 15\slash D\slash ST1\slash 00405, by Foundation for Polish Science - START scholarship, and by the Italian PRIN 2011 project
\emph{Real and complex manifolds: geometry, topology and harmonic analysis}.

\bibliography{bib1}{}

\end{document}